\documentclass[11pt,a4paper]{amsart}
\usepackage[colorlinks=true, pdfstartview=FitV, linkcolor=blue, 
citecolor=blue]{hyperref}

\usepackage{amssymb,bbm,amscd}
\usepackage{microtype,xcolor}
\usepackage{a4wide}

\theoremstyle{plain}
\newtheorem{theorem}{Theorem}[section]
\newtheorem{prop}[theorem]{Proposition}
\newtheorem{lemma}[theorem]{Lemma}
\newtheorem{coro}[theorem]{Corollary}

\theoremstyle{definition}
\newtheorem{remark}[theorem]{Remark}
 
\newtheorem{example}[theorem]{Example}
\numberwithin{equation}{section}
 
\newcommand{\dd}{\,\mathrm{d}}
\newcommand{\ts}{\hspace{0.5pt}}
\newcommand{\nts}{\hspace{-0.5pt}}

\newcommand{\cM}{\mathcal{M}}
\newcommand{\ZZ}{\mathbb{Z}}
\newcommand{\RR}{\mathbb{R}\ts}
\newcommand{\EE}{\mathbb{E}}
\newcommand{\TT}{\mathbb{T}}
\newcommand{\NN}{\mathbb{N}}
\newcommand{\XX}{\mathbb{X}}
\newcommand{\exend}{\hfill $\Diamond$}

\newcommand{\defeq}{\mathrel{\mathop:}=}

\newcommand{\myfrac}[2]{\frac{\raisebox{-2pt}{$#1$}}
      {\raisebox{0.5pt}{$#2$}}}
\newcommand{\bs}[1]{\boldsymbol{#1}}

\begin{document}

\title{Correlations of the Thue--Morse sequence}

\author{Michael Baake}
\author{Michael Coons}
\address{Fakult\"{a}t f\"{u}r Mathematik, Universit\"{a}t Bielefeld, \newline
\hspace*{\parindent}Postfach 100131, 33501 Bielefeld, Germany}
\email{mbaake@math.uni-bielefeld.de}

\address{Department of Mathematics and Statistics, California State
  University\newline \hspace*{\parindent}400 West First
  Street, Chico, California 95929, USA} \email{mjcoons@csuchico.edu}

\makeatletter
\@namedef{subjclassname@2020}{%
  \textup{2020} Mathematics Subject Classification}
\makeatother




\begin{abstract}
  The pair correlations of the Thue--Morse sequence and system are
  revisited, with focus on asymptotic results on various means.
  First, it is shown that all higher-order correlations of the
  Thue--Morse sequence with general real weights are effectively
  determined by a single value of the balanced $2$-point
  correlation. As a consequence, we show that all odd-order
  correlations of the balanced Thue--Morse sequence vanish, and that,
  for any even $n$, the $n$-point correlations of the balanced
  Thue--Morse sequence have mean value zero, as do their absolute
  values, raised to an arbitrary positive power. All these results
  also apply to the entire Thue--Morse system. We finish by showing
  how the correlations of the Thue--Morse system with general real
  weights can be derived from the balanced $2$-point correlations.
\end{abstract}

\maketitle

\centerline{Dedicated to the memory of Uwe Grimm}

\bigskip

\section{Introduction}

The study of (possibly hidden) long-range order of sequences over
finite alphabets, in particular binary ones, has a long and
interesting history; see \cite{ASbook, BGbook} and the references
therein for background. For about 100 years, starting with the insight
of Norbert Wiener, methods from harmonic analysis have been
instrumental to detect all kinds of long-range correlations via
spectral methods. While sequences with strong almost periodicity (and
hence pure point spectrum) were the first to be analysed and
understood, the ones with continuous spectra remained somewhat
enigmatic. In particular, the classic Thue--Morse (or
Prouhet--Thue--Morse) sequence with the singular continuous measure
induced by it became a paradigm of a degree of order intermediate
between pure point and absolutely continuous. First analysed by
Mahler~\cite{M1927} in 1927 by direct means, it later saw a systematic
reformulation by Kakutani \cite{Kaku} with dynamical systems methods,
and has recently been analysed in a fractal geometric setting via the
thermodynamic formalism \cite{BGKS} and in the context of
hyperuniformity \cite{BG2019}. Many obvious generalisations are known
\cite[Sec.~5.2]{BG2019}, and the progress in this direction has also
triggered new research on absolutely continuous spectra
\cite{CG2017,CGS2018,FMpre} as well as general spectral considerations
on the basis of renormalisation techniques
\cite{Luck,BS2014,BGM2018,BGM2019,BS2021}.

While much of the current literature concerns either the
autocorrelation of the Thue--Morse sequence (respectively system) or
the maximal spectral measure in the orthocomplement of the point
spectrum, compare \cite{BGbook} and \cite{Qbook}, much less is known
about the general correlation functions. It is the purpose of this
paper to fill this gap by deriving further asymptotic properties and a
general approach to the (higher-order) correlation functions and
determining some of their properties. As an added benefit, this
{provides extra insight} into the invariant probability measure on the
shift space that is induced by the Thue--Morse sequence. Let us also
mention that the (balanced) Thue--Morse sequence is Gowers uniform for
any of the standard uniformity norms \cite{Kon}, which is to say that
certain averages of $n$-point correlation functions, for $n=2^s$ with
$s\in\NN$, decay asymptotically with a power-law upper bound. It is
thus a natural question to also consider other averages of correlation
functions and their asymptotic averages, as we shall do below.
\smallskip

This paper is organised as follows. In Section~\ref{sec:prelim}, we
set the scene and give a brief summary of the Thue--Morse sequence and
{the dynamical system generated by it}, together with some classical
results on the two-point correlations (or autocorrelations). Here, we
add some results on their asymptotic properties. We continue in
Sections~\ref{sec:balanced} and \ref{sec:balanced-gen} with the
$n$-point correlations of the balanced Thue--Morse {system}, which,
via a general recurrence, are shown to be determined by the values on
an $(n\ts {-}1)$-dimensional unit cube, and further that these values
are determined by the value of the autocorrelation at zero, again, via
the recurrence. This result is then used to show that all odd-order
correlations vanish, and that, under a natural ordering, all
even-order correlations have mean value zero, in various ways. In
Section~\ref{sec:general}, we more generally show that all weighted
correlations are fully determined once the balanced ones are known,
and a general renormalisation structure is employed to achieve this.

\section{Preliminaries}\label{sec:prelim}

Let $(t_k)^{}_{k\geqslant 0}$ be the (one-sided) Thue--Morse sequence,
or word, taking the values $\pm 1$, defined by $t^{}_0=1$ and, for
$k \geqslant 0$, by
\begin{equation}\label{eq:t-rec}
   t_{2k} \, = \, t_{k} \quad \text{and} \quad
   t_{2k+1} \, = \, -t_k \ts . 
\end{equation}   
This sequence is the fixed point, starting from the seed $a$, of the
substitution
\[
  \varrho \, = \,
  \varrho_{_\mathrm{TM}} : \begin{cases} a\mapsto ab \ts , \\
    b\mapsto ba \ts ,\end{cases}
\]
where we specialise the values by $a=-b=1$. 

\begin{lemma} 
   For\/ $m\in\NN_0$, we have\/ 
   $(t_0,t_1,\ldots,t_{2^m-1})= (1, -1)^{\otimes m}$.
\end{lemma}

With $t_0 = 1$ and the the structure of the $m$-fold Kronecker
product, one obvious way to prove the lemma is through induction in
$m$ via the action of the substitution. Here, we follow an alternative
path, as it provides additional insight.

\begin{proof} 
  Note that \eqref{eq:t-rec} implies that $t_n=(-1)^{s^{}_2 (n)}$,
  where $s^{}_2 (n)$ is the number of $1$s in the binary expansion of
  $n$. To prove the claim for all $m$, we only need to show that
  $t_a=-t_{2^m+a}$ holds for all $a\in\{0,\ldots,2^{m}\nts -1\}$. This
  follows immediately since, for $a\in\{0,\ldots,2^{m}\nts -1\}$, we
  have $s^{}_2 (2^m+a)=s^{}_2 (a)+1$.
\end{proof}

Now, we properly extend the Thue--Morse sequence to a bi-infinite
sequence (or word) $w\defeq(w_n)_{n\in\ZZ}$ by defining
\[
  w_n \, = \, \begin{cases} t_n \ts ,&
         \mbox{for $n\geqslant0$} \ts ,\\
    t_{-n-1} \ts , &\mbox{for $n<0$}\ts .\end{cases}
\]
Similar to the above, the word $w$ is the bi-infinite fixed point of
the square of the Thue--Morse substitution, $\varrho^2$, starting from
the seed $a|a$, where as before $a=-b=1$; see \cite[Rem.~4.8]{BGbook}.

Let $S$ be the shift operator, $(Sw)_i=w_{i+1}$, and
$\XX=\XX(w)\defeq\overline{\{S^iw:i\in\ZZ\}} \subset \{ \pm \ts 1
\}^{\ZZ}$ be the (discrete) hull. The space $\XX$ together with the
$\ZZ$-action of the shift forms a topological dynamical system,
denoted by $(\XX, \ZZ)$. It admits precisely one invariant probability
measure, say $\mu$. In other words, $(\XX, \ZZ)$ is uniquely
ergodic. Since $\XX$ is also minimal (meaning that the $\ZZ$-orbit of
every element is dense in $\XX$), the system is even \emph{strictly
  ergodic}; see \cite{Kaku,BGbook} for details. Using this measure
$\mu$, one has
\[
  \int_{\XX} x_{m^{}_0} x_{m^{}_1} \cdots
  x_{m^{}_{n-1}}\dd\mu(x) \, = \lim_{N\to\infty}
  \myfrac{1}{N} \sum_{k=0}^{N\nts -1} y^{}_{k+m^{}_0}
  y^{}_{k+m^{}_1} \nts \cdots \ts y^{}_{k+m^{}_{n-1}} \ts ,
\]
where the choice of $y=(y^{}_i)^{}_{i\in\ZZ}\in\XX$ is arbitrary. This
equality is a consequence of Birkhoff's ergodic theorem for uniquely
ergodic shift spaces, because the right-hand side is the Birkhoff
average of a continuous function on $\XX$.

In this paper, we are interested in the \emph{correlations} of the
Thue--Morse {system}, both for the standard balanced weights described
above and for more general real weights. Since the measure $\mu$ is
shift invariant, without loss of generality, we can fix one of the
$m^{}_i$; we thus set $m^{}_0=0$. Further, we let
$f\colon \{-1,1\}\longrightarrow \RR$ and define the \emph{general}
$n$-point correlations of the $f$-weighted Thue--Morse sequence {(and
  system)} by
\begin{equation}\label{eq:eta-f}
  \eta^{}_{f} (m^{}_1 , m^{}_2 , \ldots,
  m^{}_{n-1}) \, \defeq \int_{\XX}f(x^{}_0)\,
  f ( x^{}_{m^{}_{1}} ) \cdots f ( x^{}_{m^{}_{n-1}} ) \dd\mu(x) \ts .
\end{equation}
In the balanced case, that is, when $f =\mathrm{id} $ is the identity
function, we suppress the subscript $f$ and simply write $\eta$
instead of $\eta^{}_{\mathrm{id}}$. For a study towards a different
generalisation, using multiple weight functions, see Aloui
\cite{A2022}.

We note that $\mu$ is the \emph{patch frequency measure} of the
system, as defined by its values on the cylinder sets defined via all
finite words. Their frequencies can be extracted from the
frequency-normalised Perron--Frobenius eigenvectors of the induced
substitution matrices for Thue--Morse words of length $n$; see
\cite[Sec.~5.4.3]{Qbook} or \cite[Sec.~4.8.3]{BGbook} for details. The
frequency module of the Thue--Morse system (that is, the Abelian group
generated by all occurring frequencies) is given by
\begin{equation}\label{eq:fm}
  \cM_{\mu} \, = \, \Big\{ \myfrac{m}{3\cdot2^r} :
  r\in\NN_0, m\in\ZZ \Big\} .
\end{equation}
In particular, all word frequencies are integer linear combinations of
letter frequencies and frequencies of words of length $2$---which
gives one way to prove \eqref{eq:fm}. Indeed, the single $3$ in the
denominator emerges from the frequencies of words of length $2$, while
the powers of $2$ reflect the substitution structure. One can then
check that $\cM_{\mu}$, in the parametrisation used, is an Abelian
group, and the smallest one that contains all word frequencies.

While this view is, in some ways, satisfactory, it is still incomplete
in the sense that the calculation of the frequencies is not
trivial. This emphasises an alternative viewpoint via the
correlations, which also completely determine the measure $\mu$
because they comprise the patch frequencies via suitable choices of
the weight function $f$ and the number of points in \eqref{eq:eta-f}.
This is one of our motivations to study the correlations $\eta^{}_{f}$
in some generality.

\section{Pair correlations for balanced weights}\label{sec:balanced}

Let us consider the correlations for balanced weights,
$\{ \pm \ts 1 \}$.  The standard two-point correlation coefficients
$\eta(m)$ of the Thue--Morse sequence {(and system)}, which are also
known as the autocorrelation coefficients, are usually introduced as
\[
  \eta(m) \, = \lim_{N\to\infty}\myfrac{1}{N}
  \sum_{k=0}^{N\nts -1} t^{}_k \ts t^{}_{k+m} \ts ,
\]
which is consistent with our above definition. By symmetry, one has
$\eta(-m)=\eta(m)$, which follows easily after dropping finitely many
terms from the sum. Further, we clearly get $\eta(0)=1$, and, for
$m\geqslant 0$, one finds the repeatedly derived recursions
\cite{M1927,Kaku,BGbook}
\begin{equation}\label{eq:v2recs}
\begin{split}
  \eta(2m) \, & = \, \eta(m) \qquad\mbox{and} \\
  \eta(2m+1) \, & = -\myfrac{1}{2}\big(\eta(m)+\eta(m+1)\big) ,
\end{split}  
\end{equation}
which are a direct consequence of the substitution structure.  These
recurrences allow one to compute all of the values of $\eta$ from
$\eta(0)$. In particular, one has
\begin{equation}\label{eq:initial}
  \eta(1) \, = \, - \frac{\eta (0)}{3} \, = \,
  -\myfrac{1}{3}\ts ,
\end{equation}
by solving the second equation in \eqref{eq:v2recs} with $m=0$ for
$\eta (1)$. We can write the pair of recurrence equations
\eqref{eq:v2recs} in matrix form as
\[
  \left(\begin{matrix} \eta(2m)\\
      \eta(2m+1)\end{matrix}\right)=\myfrac{1}{2}\left(\begin{matrix} 2
      & 0\\ -1 &-1 \end{matrix}\right)\left(\begin{matrix} \eta(m)\\
      \eta(m+1)\end{matrix}\right),
\]
which is valid for $m\geqslant 0$. As this rational matrix will be
important for us later, we record a few points of interest
here. Firstly, it has eigenvalues $1$ and $-1/2$, with right
eigenvectors $(1,-1/3)^T$ and $(0,1)^T$, respectively. Next, set
\begin{equation}\label{eq:E-def}
  \bs{E}^{}_0 \, \defeq \, \begin{pmatrix} 1 & 0\\ 0
    &-1 \end{pmatrix}, \quad
     \bs{E}^{}_1 \, \defeq \, \begin{pmatrix}
      0 & -1\\ 0  &1 \end{pmatrix}\quad
    \mbox{and}\quad\bs{J} \, \defeq \,
    \begin{pmatrix} 0 & 1\\ 1 &0 \end{pmatrix} .
\end{equation}
Note that $\bs{E}^{}_0$ and $\bs{J}$ are involutions, while
$\bs{E}^{}_{1}$ is an idempotent. For a $2{\times}2$ matrix $\bs{A}$,
define
$\bs{A}' \defeq \bs{J}\! \bs{A} \ts \bs{J}^{-1} = \bs{J} \! \bs{A} \ts
\bs{J}$. Then, one has $\bs{A}''=\bs{A}$, while
$\bs{E}_0' = -\bs{E}^{}_0$.  In particular, this yields the
decomposition
\[
     \left(\begin{matrix} 2 & 0\\ -1
      &-1 \end{matrix}\right) \, = \,
    \bs{E}^{}_0+\bs{J}\nts\bs{E}^{}_1 \ts
    \bs{J} \, = \, \bs{E}^{}_0+\bs{E}_1'  .
\]

\begin{remark}\label{rem:reno}
  The recursion relations \eqref{eq:v2recs} define an infinite set of
  linear equations for the numbers $\eta (m)$ with $m\in\NN_0$. This
  set contains a (maximal) finite subset of equations that is closed,
  in the sense that they are equations for finitely many coefficients
  among themselves (an no others), while all remaining coefficients
  are then fully determined recursively from these ones. Here, they
  are the two equations for $m=0$, namely
\[
    \eta (0) \, = \, \eta (0) \quad\text{and}\quad
    \eta (1) \, = \, -\myfrac{1}{2}\big(\eta(0)+\eta(1)\big) .
\]   
The first one is a tautology, and simply an indication that the
recursion relations alone do not specify the value of $\eta (0)$. The
second specifies $\eta (1)$ as a function of $\eta (0)$, as we
saw in \eqref{eq:initial}, while all $\eta (m)$ with $m \geqslant 2$
are then determined recursively. In other words, the linear solution
space to \eqref{eq:v2recs} is one-dimensional, and once $\eta (0)$ is
given, $\eta$ is completely specified.
   
This structure is quite typical for a set of renormalisation
equations. The infinite set of equations contains a (maximal) finite
subset that is closed, often called the \emph{self-consistency part},
while all others quantities are fixed recursively. This structure has
recently been identified in more general inflation tilings, both in
one and in higher dimensions, and gives access to various spectral
properties; see \cite{BG,BGM2019} and references therein.  \exend
\end{remark}

\begin{remark}
  The coefficients $\eta(m)$ from \eqref{eq:v2recs} define a function
  $\eta\colon\ZZ\longrightarrow\RR$ which is positive definite and
  defines a measure $\gamma=\eta\cdot\delta_{\ZZ}$ that is a positive
  definite pure point measure on $\ZZ$. As such, it is Fourier
  transformable, which gives the positive measure
\[
  \widehat{\gamma} \, = \, \nu_{_\mathrm{TM}}*\delta^{}_{\ZZ} \ts ,
\]
where $\nu^{}_{_\mathrm{TM}}$ is a purely singular continuous
probability measure on the torus $\TT=\RR/ \ZZ$. Using the standard
unit interval as a model for $\mathbb{T}$, one obtains the Riesz
product representation
\[
   \nu^{}_{_\mathrm{TM}}  \, = \,
   \prod_{\ell=0}^\infty\big(1-\cos (2^{\ell+1}\pi x)\big) ,
\]
which converges weakly. Here, the right-hand side is the standard
notation for a sequence of absolutely continuous measures, each given
by its Radon--Nikodym density, with the limit being singular
continuous; see \cite[Ch.~10.1]{BGbook} and references therein for
details. In particular, the pointwise limit of the right-hand side
vanishes on a set of full measure, while it diverges, or is
ill-defined, on an uncountable null set. The details of such measures
are best studied via the thermodynamic formalism \cite{BGKS}.
  
The proof of the singular continuity of $\nu_{_\mathrm{TM}}$ rests
upon two properties of $\eta$. First, by showing
\begin{equation}\label{eq:Wiener}
     \lim_{N\to\infty} \myfrac{1}{N} \sum_{m=0}^{N\nts -1} \eta (m)^2
     \, = \, 0 \ts ,
\end{equation}
compare \cite[Lemma~10.2]{BGbook}, one excludes pure point
contributions, via Wiener's criterion. Then, since
$\eta (2^{\ell}) = -1/3$ for all $\ell\in\NN$, one can employ the
Riemann--Lebesgue lemma to show that no absolutely continuous
component can exist; see \cite{Kaku,BGbook} for the details.  \exend
\end{remark}

Beyond \eqref{eq:Wiener}, also the mean vanishes asymptotically,
\begin{equation}\label{eq:mean}
    \lim_{n\to\infty} \myfrac{1}{N} \sum_{m=0}^{N\nts -1} \eta (m)
     \, = \, 0 \ts .
\end{equation}
To see why this is true, set
$\Sigma (N) = \frac{1}{N} \sum_{m=0}^{N\nts -1} \eta (m)$ and observe
that, as $N\to\infty$, one has
$\Sigma (2N+1) = \frac{2N}{2N+1} \Sigma(2N) + O (N^{-1})$.  Now, since
$\lvert \eta (m)\rvert \leqslant 1$ for all $m\in\NN_0$, one finds
\[
\begin{split}
   \Sigma (2N) \, & = \, \myfrac{1}{2} \biggl( \myfrac{1}{N}
   \sum_{m=0}^{N\nts -1} \eta (2m) \, + \, \myfrac{1}{N}
   \sum_{m=0}^{N\nts -1} \eta (2m+1) \biggr) \\[1mm]
   & = \, \myfrac{1}{2} \biggl( \myfrac{1}{N}
   \sum_{m=0}^{N\nts -1} \eta (m) \, - \, \myfrac{1}{2N}
   \sum_{m=0}^{N\nts -1} \bigl( \eta (m) + \eta (m+1) \bigr) \biggr) \\[2mm]
   & = \, \frac{\Sigma (N)}{2} - \frac{\Sigma (N)}{4}
   - \frac{\Sigma (N)}{4} + \frac{\eta (0) - \eta (N)}{4N}\\[3mm] 
   & = \, \frac{1-\eta(N)}{4N} \, = \,     O (N^{-1})
\end{split}
\]
as $N\to\infty$, where \eqref{eq:v2recs} was used in the second line,
and \eqref{eq:mean} is obvious from here. In fact, the above
derivation gives the stronger relation
\[
  \sum_{m=0}^{2N\nts -1} \eta (m) \, = \, \frac{1-\eta (N) }{2}
  \, \in \, \bigl[ 0, \tfrac{2}{3} \bigr] .
\]
Further, using the result from \eqref{eq:v2recs} and \eqref{eq:Wiener}
in conjunction with $|\eta(n)| \leqslant 1$, one finds the following
consequence.

\begin{prop}\label{prop:eta-means}
  For\/ $k\in\NN_0$, one has\/
  $\,\lim_{N\to\infty} \frac{1}{N} \sum_{m=0}^{N\nts -1}\, \eta(m)^{k}
  = \delta^{}_{k,0}\ts $.  Furthermore, one has\/
  $\,\lim_{N\to\infty} \frac{1}{N} \sum_{m=0}^{N\nts -1}\,
  |\eta(m)|^{k} = 0$ for all\/ $k \geqslant 2$.  \qed
\end{prop}

Now, define $\mu^{}_{\pm} = \frac{1 \pm \eta}{2}$, which means
$\mu^{}_{+} (n) + \mu^{}_{-}(n) = 1$ and
$\mu^{}_{+} (n) - \mu^{}_{-}(n) = \eta(n)$ for all $n\in\NN_0$.  From
Proposition~\ref{prop:eta-means}, one then finds
\begin{equation}
  \lim_{N\to\infty} \myfrac{1}{N} \sum_{m=0}^{N\nts - 1}
  \mu^{}_{\pm} (m)^{k} \, = \, 2^{-k} 
\end{equation}
for $k\in\NN_0$. In fact, one obtains the coupled recursion relations
\[
  \mu^{}_{\pm} (2m) \, = \, \mu^{}_{\pm} (m)
  \quad \text{and} \quad \mu^{}_{\pm} (2m+1) \, = \, \myfrac{1}{2}
  \bigl( \mu^{}_{\mp} (m) + \mu^{}_{\mp} (m+1) \bigr)  
\]
for $m\in\NN_0$. These also follow from the general renormalisation
relations in \cite[Eq.~(16)]{BG} by observing the letter exchange
symmetry of the Thue--Morse {system} under $a \leftrightarrow b$.
This connection has the following immediate consequence.

\begin{coro}
  If\/ $\nu^{}_{\alpha \beta} (n)$ with\/
  $\alpha, \beta \in \{ a, b\}$ and\/ $n\in\ZZ$ denotes the relative
  frequency of occurrence of the distance\/ $n$ in the Thue--Morse
  sequence\/ $w$ between a letter of type\/ $\alpha$ to the left and a
  letter of type\/ $\beta$ to the right, with the obvious inversion
  for negative\/ $n$, one has
\[   
   \nu^{}_{aa} (n)  \, = \, \nu^{}_{bb} (n) \, = \,
    \mu^{}_{+} (\lvert n \rvert ) \, = \, 
    \frac{ 1+\eta (\lvert n \rvert) }{2} 
   \quad \text{and} \quad 
   \nu^{}_{ab} (n) \, = \, \nu^{}_{ba} (n) \, = \, 
   \mu^{}_{-} (\lvert n \rvert ) \, = \,
   \frac{1 - \eta (\lvert n \rvert )}{2}
\]
  for all\/ $n\in\ZZ$.    \qed
\end{coro}

It remains to understand the mean value of $|\eta(m)|$. Via the
Cauchy--Schwarz inequality in conjunction with $|\eta(m)|\leqslant 1$,
it is elementary to derive the relation
\[
  \biggl( \myfrac{1}{N} \sum_{m=0}^{N\nts - 1} |\eta (m)|
  \biggr)^2 \, \leqslant \, \myfrac{1}{N} \sum_{m=0}^{N\nts - 1}
  \eta(m)^2 \, \leqslant \, \myfrac{1}{N} \sum_{m=0}^{N\nts - 1}
  |\eta (m)| \ts ,
\]
so that \eqref{eq:Wiener} implies
$\,\lim_{N\to\infty} \frac{1}{N} \sum_{m=0}^{N\nts -1} |\eta(m)| = 0$,
as expected. In fact, not only is the mean value of $|\eta(m)|$ equal 
to zero, we can say more as follows. 

\begin{theorem}\label{thm:meanalpha} 
  For any\/ $\alpha> \frac{\log (3)}{\log (4)} \approx
  0.792{\ts}481{\ts}505 $, we have
\[
   \lim_{N\to\infty} \myfrac{1}{N^{\alpha}} 
   \sum_{m=0}^{N\nts - 1}   |\eta(m) |  \, = \,  0 \ts .
\]
\end{theorem} 

\begin{proof}
To start, note that, when $N\in[2^{2\ell},2^{2\ell+2}]$,
  one has the estimate
\[
  \myfrac{1}{N^{\alpha}}\sum_{m=0}^{N\nts -1}|\eta (m)|
  \, \leqslant \,
  \myfrac{1}{(2^{2\ell})^{\alpha}}\sum_{m=0}^{2^{2\ell+2}-1}|\eta (m)|
  \, = \,
  \myfrac{4^{\alpha}}{(2^{2\ell+2})^{\alpha}}\sum_{m=0}^{2^{2\ell+2}-1}|
  \eta (m)| \ts ,
\]
which implies that it suffices to show
$2^{-(2\ell+2){\alpha}}\sum_{m=0}^{2^{2\ell+2}-1}|\eta (m)| = o ( 1 )$
as $\ell \to \infty$.

To this end, note that the recurrences \eqref{eq:v2recs} imply the
estimates
\begin{equation}\label{eq:esti}
\begin{split}
   |\eta(4m)| \, & = \, |\eta(m)|  \ts , \\[1mm]
   |\eta(4m+1)| \, & = \, \myfrac{1}{2}\big|
   \eta(2m)+\eta(2m+1)\big| \, = \, \myfrac{1}{2}
   \big| \eta(m)-\myfrac{1}{2}\big(\eta(m)+\eta(m+1)\big) \big|  \\[1mm]
  &\qquad \quad \leqslant \, \myfrac{1}{4} |\eta(m)|+
     \myfrac{1}{4}|\eta(m+1)| \ts ,  \\[1mm]
  |\eta(4m+2)| \, & = \, |\eta(2m+1)| \, \leqslant \,
    \myfrac{1}{2} |\eta(m)| + \myfrac{1}{2} |\eta(m+1)| \ts , \\[1mm]
  |\eta(4m+3)| \, & = \, \myfrac{1}{2} \big| \eta(2m+1)+
       \eta(2m+2)\big| \, = \, \myfrac{1}{2} \big|
       \eta(m+1) -\myfrac{1}{2} \big(\eta(m)+\eta(m+1)\big)\big| \\[1mm]
  &\qquad \quad \leqslant \, \myfrac{1}{4} |\eta(m)|+
       \myfrac{1}{4} |\eta(m+1)| \ts .
\end{split} 
\end{equation}
As in the statement of the theorem, we let
$\alpha> \frac{\log (3)}{\log (4)}$ and set
$\Sigma' (N) = \frac{1}{N^{\alpha}} \sum_{m=0}^{N\nts -1} |\eta
(m)|$. Arguing as above in \eqref{eq:esti}, we have
{\allowdisplaybreaks
\begin{align*}
     \Sigma' (4N) \, & = \, \myfrac{1}{(4N)^{\alpha}} \sum_{r=0}^{3}
   \sum_{m=0}^{N\nts -1} |\eta (4m + r)|  \\[2mm]
   & \leqslant \, \myfrac{1}{4^{\alpha}} \biggl( \myfrac{1}{N^{\alpha}}
   \sum_{m=0}^{N\nts -1} |\eta (m)| \, + \, \myfrac{1}{4N^{\alpha}}
   \sum_{m=0}^{N\nts -1} \bigl( |\eta (m)| + |\eta (m+1)| \bigr)\\[1mm]
   &\qquad\quad + \, \myfrac{1}{2N^{\alpha}}
   \sum_{m=0}^{N\nts -1} \bigl( |\eta (m)| + |\eta (m+1)| \bigr) 
        \, + \, \myfrac{1}{4N^{\alpha}} \sum_{m=0}^{N\nts-1} 
   \bigl( |\eta (m)| + |\eta (m+1)| \bigr) \biggr) \\[2mm]
   & = \, \myfrac{3}{4^{\alpha}} \,  \Sigma' (N) + 
   \frac{|\eta (N)| - |\eta (0)|}{4N^{\alpha}} \, \leqslant \,
    \myfrac{3}{4^{\alpha}} \, \Sigma' (N) \ts ,
\end{align*}}
since $|\eta(N)|-|\eta(0)|=|\eta(N)|-1\leqslant 0$. Now, letting 
$N=2^{2\ell+2}$, for $\ell\geqslant 0$, we get
\begin{equation*}
    \frac{1}{2^{(2\ell+2)\alpha}} \sum_{m=0}^{2^{2\ell+2}-1} 
    |\eta (m)| \, = \, \Sigma'(2^{2\ell+2}) \, \leqslant \,
    \myfrac{3}{4^{\alpha}} \, \Sigma' (2^{2\ell})
    \, \leqslant \, \Bigl(\myfrac{3}{4^{\alpha}}\Bigr)^{\ell+1} 
    = \, o(1) \ts . 
\qedhere\end{equation*}
\end{proof}

\begin{remark} 
  The lower bound of $\log(3)/\log (4)$ {from}
  Theorem~\ref{thm:meanalpha} is reminiscent of the result that {there
    is a $C>0$ such that}
\[
    \sup_{\theta\in[0,1]}\,\biggl|\sum_{k\leqslant x}
    \mathrm{e}^{2\pi \mathrm{i} k \theta} t^{}_k \biggr|
    \, \leqslant \, C x^{\log(3)/\log(4)}, 
\]
wherein the exponent $\log(3)/\log (4)$ is optimal; see
\cite{G1968,N1969,NS1975}.  At first glance, one may conjecture that
this exponent is also optimal in Theorem~\ref{thm:meanalpha}, but this
is not the case. In fact, if, instead of partitioning the positive
integers into arithmetic progressions with common difference $4$ in
\eqref{eq:esti}, we use common difference $8$, we have that
Theorem~\ref{thm:meanalpha} holds for any
$\alpha>\log(5)/\log(8)\approx0.773 {\ts} 976 {\ts} 031 {\ts} 3.$
Using larger common differences, the bound for $\alpha$
lowers. Computationally, we have gone up to common difference
$2^{20}$, which shows that one may take any
$\alpha > 0.652 {\ts} 632 {\ts} 647 {\ts} 6$ in
Theorem~\ref{thm:meanalpha}.  We leave the determination of the
optimal value of $\alpha$ for further investigation.  \exend
\end{remark}

Next, using Theorem \ref{thm:meanalpha} with the value $\alpha=1$ 
gives us the following result.

\begin{theorem}\label{thm:means}
For every real\/ $\beta>0$, 
\[
   \lim_{N\to\infty} \myfrac{1}{N} 
   \sum_{m=0}^{N\nts - 1}  |\eta(m) |^{\beta} \, = \, 0 \ts .
\]   
\end{theorem}

\begin{proof} 
H\"{o}lder's inequality gives 
\begin{equation}\label{eq:holder}
    \biggl(\myfrac{1}{N}\sum_{m=0}^{N\nts -1 }
    |\eta(m)|^{\beta}\biggr)^{1+\beta} \, \leqslant \;
     \biggl( \, \sum_{m=0}^{N\nts -1 }
     |\eta(m)|^{1+\beta} \biggr)^{\beta} \,
     \sum_{m=0}^{N\nts -1 }\myfrac{1}{N^{1+\beta}}
     \: = \, \biggl( \myfrac{1}{N}\sum_{m=0}^{N\nts -1 }
     |\eta(m)|^{1+\beta}\biggr)^{\beta}.
\end{equation}
Since $|\eta(m)|\leqslant 1$ for all $m$, we have
$|\eta(m)|^{1+\beta}\leqslant |\eta(m)|$ for all $\beta\geqslant 0$.
So, for $\beta>0$, the right-hand side (and then also the left-hand
one) of \eqref{eq:holder} limits to zero as $N$ grows.
\end{proof}

We are now set to consider higher-order correlation functions.

\section{General correlations for balanced
  weights}\label{sec:balanced-gen}

Analogously, we define the $n$-point correlations of the Thue--Morse
system by
\[
  \eta (m^{}_1 ,m^{}_2,\ldots, m^{}_{n-1}) \, =
  \int_{\XX} x^{}_0 \cdot  (S^{\ts m^{}_1}x)^{}_0
  \cdots (S^{\ts m^{}_{n-1}}x)^{}_0 \dd \mu(x) \ts .
\]
Similar to the above, these $n$-point correlations can be computed
recursively.

\begin{prop}\label{prop:reduction} 
  For each\/ $i\in\{1,\ldots,n\ts {-}1\}$, let\/ $r^{}_i \in\{0,1\}$
  and set\/ $r = r^{}_1+\cdots+r^{}_{n-1}$. Then, for any integers\/
  $m^{}_1, \ldots, m^{}_{n-1} \geqslant 0$, we have
\[
  \begin{split}
    \eta(2 m^{}_1 + r^{}_1, \ldots, & 2 m^{}_{n-1} + r^{}_{n-1})
     \, = \, \\[1mm]
    & \frac{(-1)^r}{2} \big( \eta (m^{}_1 , \ldots, m^{}_{n-1})
     + (-1)^{n} \eta ( m^{}_1 + r^{}_1 , \ldots, 
     m^{}_{n-1} + r^{}_{n-1}) \big) .
  \end{split}
 \]     
\end{prop}

\begin{proof} 
  Observe that, by Birkhoff's ergodic theorem in conjunction with the
  absolute convergence and hence rearrangement invariance of all
  involved sums, we have
\begin{align*}
  \eta( & 2 m^{}_1 + r^{}_1 , \ldots, 2m^{}_{n-1} + r^{}_{n-1}) 
  \, = \lim_{N\to\infty}\frac{1}{2N}\sum_{k=0}^{2N\nts -1}
    t^{}_k \ts t^{}_{k+2m^{}_1 + r^{}_1} \! \cdots \ts
    t^{}_{k+2m^{}_{n-1} + r^{}_{n-1}} \\[1mm]
 &=\lim_{N\to\infty}\myfrac{1}{2N}\sum_{j=0}^{N\nts -1}
 \bigl( t^{}_{2j} \ts t^{}_{2j+2m^{}_1 + r^{}_1}\!\cdots\ts
  t^{}_{2j+2m^{}_{n-1} + r^{}_{n-1}}  \! + 
  t^{}_{2j+1}\ts t^{}_{2j + 2m^{}_1 + r^{}_1 + 1}\!\cdots\ts 
  t^{}_{2j+2m^{}_{n-1} + r^{}_{n-1}+1}\bigr) \\[1mm]
 &\nts \!\!\stackrel{\eqref{eq:t-rec}}{=}\!\!\nts
  \lim_{N\to\infty}\frac{(-1)^r}{2N}
  \sum_{j=0}^{N\nts -1} \bigl( t^{}_{j}\ts t^{}_{j+m^{}_1}\!\cdots\ts 
  t^{}_{j+m^{}_{n-1}} + (-1)^{n} \ts
  t^{}_{j} \ts t^{}_{j+m^{}_1 + r^{}_1}\!\cdots\ts 
  t^{}_{j+m^{}_{n-1} + r^{}_{n-1}} \bigr) \\[2mm]
 &=\frac{(-1)^r}{2} \big( \eta(m^{}_1 , \ldots , m^{}_{n-1}) +
    (-1)^{n} \ts \eta( m^{}_1 + r^{}_1 , \ldots , 
    m^{}_{n-1} + r^{}_{n-1})\big),
\end{align*}
which completes the argument.
\end{proof}

\begin{remark} 
  The special case of Proposition~\ref{prop:reduction} concerning
  $4$-point correlations was proved by Ong. See \cite{Opre} for this
  identity, and for some beautiful pictures using the values of the
  $4$-point correlations of the Thue--Morse sequence.  \exend
\end{remark}

\begin{remark}
  It is well known that the \emph{period doubling} (pd) substitution
  $\varrho \colon a \mapsto ab \ts , \, b \mapsto aa$ defines a
  dynamical system that is a factor of the Thue--Morse system, but
  displays pure point spectrum and hence a higher degree of order; see
  \cite[Thm.~4.7]{BGbook} for the details of the $2:1$ covering
  relationship.  Now, setting $a=-b=-1$, the pair correlations of the
  period doubling system with these weights appear here as a subset,
  via
\[
    \eta^{}_{\mathrm{pd}} (m) \, = \, 
    \eta (1, m, m \ts {+}1)
\]
for $m\in\ZZ$. This demonstrates that correlation functions with
singular continuous averaging behaviour can still display perfect,
almost-periodic order on thin subsets. This was also discussed in
\cite{Aernout}. The existence of these highly structured thin subsets
suggest that comparing partial sums of higher-order Thue--Morse
correlations with $N$ will not be enough, and that a higher power of
$N$ is necessary, a phenomenon we describe in what follows.  \exend
\end{remark}

Proposition~\ref{prop:reduction} is the multi-dimensional analogue of
the recursions in \eqref{eq:v2recs}. As above, in the case $n=2$, this
generalisation can be used to prove a zero-mean-value result analogous
to \eqref{eq:mean}, here over the nested $(n{-}1)$-dimensional integer
cubes in the first orthant. That is,
\begin{equation}\label{eq:nmeanvalue}
     \lim_{N\to\infty}\, \myfrac{1}{N^{n-1}}
     \sum_{0\leqslant m^{}_1, \ldots ,m^{}_{n-1} \leqslant N\nts - 1}
     \eta(m^{}_1 , \ldots , m^{}_{n-1}) \, = \, 0 \ts .
\end{equation} 
To see this, we proceed as above by setting
$\mathfrak{S}(N)=\frac{1}{N^{n-1}} \sum_{0\leqslant m^{}_1, \ldots,
  m^{}_{n-1} \leqslant N\nts -1} \eta(m^{}_1 , \ldots,
m^{}_{n-1})$. Again, note that
$\mathfrak{S}(2N{+}1)=
\frac{(2N)^n}{(2N+1)^n}\mathfrak{S}(2N)+O(N^{-1})$ as $N\to\infty$,
where the error term is the result of the number of points on the
surface of the cube growing like $N^{n-1}$.  It thus suffices to show
that $\mathfrak{S}(2N)\longrightarrow 0$ as $N\to\infty$. This, using
Proposition~\ref{prop:reduction}, follows from
\begin{align*} 
    \mathfrak{S}(2N) \, & = \, 
    \myfrac{1}{(2N)^{n-1}}\sum_{0\leqslant m^{}_1 , \ldots, m^{}_{n-1}
    \leqslant 2N\nts -1} \eta(m^{}_1 ,\ldots, m^{}_{n-1}) \\[1mm]
   & = \, \myfrac{1}{(2N)^{n-1}}
   \sum_{0 \leqslant m^{}_1 , \ldots , m^{}_{n-1} \leqslant N\nts -1}
   \;\sum_{r^{}_1 , \ldots , r^{}_{n-1} \in\{0,1\}}
   \eta(2 m^{}_1 + r^{}_1 , \ldots , 2 m^{}_{n-1} + r^{}_{n-1})\\[1mm]
   & =  \, \myfrac{1}{2(2N)^{n-1}} \sum_{r^{}_1 , \ldots , r^{}_{n-1} \in\{0,1\}}
    (-1)^{r^{}_1 + \ldots + r^{}_{n-1}} \\[1mm]
   &\qquad\qquad\times
     \sum_{0\leqslant m^{}_1 , \ldots , m^{}_{n-1} \leqslant N\nts -1}
     \big(\eta(m^{}_1 , \ldots , m^{}_{n-1}) + (-1)^{n}\ts
     \eta(m^{}_1 + r^{}_1 , \ldots , m^{}_{n-1} + r^{}_{n-1} )\big)\\[1mm]
   &= \, \frac{(-1)^{n}}{2(2N)^{n-1}}
     \sum_{r^{}_1 , \ldots , r^{}_{n-1} \in \{0,1\}}
     (-1)^{r^{}_1 + \ldots + r^{}_{n-1}} \\[1mm]
   &\qquad\qquad\times
     \sum_{0\leqslant m^{}_1 , \ldots , m^{}_{n-1} \leqslant N\nts - 1}
     \eta(m^{}_{1}+ r^{}_{1},\ldots,m^{}_{n-1}+ r^{}_{n-1}),
\end{align*} 
where, for the fourth equality, we have used that
$\sum_{r^{}_1 ,\ldots,r^{}_{n-1}\in\{0,1\}}(-1)^{r^{}_1 + \ldots +
  r^{}_{n-1}}=0$.  Now, since
$|\eta(m^{}_1 , \ldots , m^{}_{n-1})| \leqslant 1$, we have
\[ 
\begin{split}
    \sum_{0\leqslant m^{}_1 , \ldots , m^{}_{n-1} \leqslant N\nts - 1} \!    
    & \eta(m^{}_1 + r^{}_1 , \ldots , m^{}_{n-1} + r^{}_{n-1}) \\
    & =
    \sum_{0\leqslant m^{}_1 , \ldots , m^{}_{n-1} \leqslant N\nts -1}
    \! \eta(m^{}_1 , \ldots , m^{}_{n-1} ) \, + \, O(N^{n-2}),
\end{split}
\] 
so that $\mathfrak{S}(2N)$, as $N\to\infty$, is equal to 
\[
   \frac{(-1)^{n}}{2(2N)^{n-1}}\sum_{r^{}_1 , \ldots , r^{}_{n-1} \in\{0,1\}}
   \! (-1)^{r^{}_1 + \ldots + r^{}_{n-1}} \biggl(\;
   \sum_{0 \leqslant m^{}_1 , \ldots , m^{}_{n-1} \leqslant N\nts - 1} \!
   \eta(m^{}_1 , \ldots , m^{}_{n-1}) \, + \, O(N^{n-2})\! \biggr),
\]
which overall is $O(N^{-1})$.

To compute all of the $n$-point correlations for a given $n$, we only
need to know the values of $\eta$ at the $2^{n-1}$ corners of the
$(n\ts {-}1)$-dimensional unit hypercube, because all other
correlations are recursively determined from these finitely many
values. They can be calculated in two different ways as follows.

First, for a point $(r^{}_1 , \ldots , r^{}_{n-1} )\in\{0,1\}^n$, with
$r=r^{}_1 + \ldots + r^{}_{n-1}$ as above, we get
\begin{equation}\label{eq:01}
\begin{split}
  \eta(r^{}_1 , r^{}_2 , \ldots , r^{}_{n-1}) \, & =
  \lim_{N\to\infty}\myfrac{1}{N}\sum_{k=0}^{N\nts -1}
  t^{}_k \ts t^{}_{k+r^{}_1}\!\cdots\ts t^{}_{k+r^{}_{n-1}} \\[1mm]
  & = \,\begin{cases} 0, & \mbox{if $n$ is
      odd,}\\ 1, &\mbox{if $n$ is even and $r$ is even,}\\
    -\frac{1}{3},&\mbox{if $n$ is even and $r$ is
      odd},\end{cases}
\end{split}      
\end{equation} 
which follows simply from the fact that $(t_m)^2 =1$ for any $m$ in
conjunction with the fact that $1$ and $-1$ occur equally frequently
and with bounded gaps in $(t^{}_{k})^{}_{k\in\NN_0}$. In particular,
one has
\[
    \lim_{N\to\infty} \myfrac{1}{N} \sum_{k=n}^{n+N\nts -1} t^{}_{k}
    \, = \, 0 \ts,
\]
uniformly in $n\in\NN_0$. Combining this with
Proposition~\ref{prop:reduction} gives the following immediate
consequence.

\begin{coro}\label{coro:odd}
  All odd-order correlations of the balanced Thue--Morse {system} 
  vanish. \qed
\end{coro}

The second approach generalises Remark~\ref{rem:reno} in realising
that the recursions in Proposition~\ref{prop:reduction} can once again
be seen as an infinite set of linear equations for the coefficients
$\eta ( m^{}_{1}, \ldots, m^{}_{n-1} )$. It is clear that the
$2^{n-1}$ equations with all $m_i \in \{ 0, 1 \}$ form a closed
subset, {as explained in Remark~\ref{rem:reno}}, and that all other
coefficients are recursively determined from the solution of these
equations, which are the $\eta$-coefficients at the $2^{n-1}$ corners
of the unit hypercube.

\begin{coro}\label{coro:eta0}
  The solution space of the linear recursion equations from
  Proposition~\textnormal{\ref{prop:reduction}} is
  one-dimensional. Consequently, for any\/ $n\in \NN$, the\/ $n$-point
  correlations of the balanced Thue--Morse {system} are uniquely
  specified by a single number, namely by\/ $\eta (0, \ldots , 0)$.
\end{coro}

\begin{proof}
  The recursive structure is clear from
  Proposition~\ref{prop:reduction}.  In particular, all coefficients
  are fully determined once the
  $\eta (r^{}_{1}, \ldots , r^{}_{n-1} )$ for all
  $r^{}_i \in \{ 0,1 \}$ are known. Choosing all $m_i = 0$ in the
  recursion relations of Proposition~\ref{prop:reduction}, one gets
\[
   \eta (r^{}_{1}, \ldots , r^{}_{n-1} ) \, = \,
   \frac{(-1)^r}{2} \bigl( \eta (0, \ldots, 0) +
   (-1)^{n} \ts \eta ( r^{}_{1}, \ldots , r^{}_{n-1} )\bigr)
\] 
with $r = r^{}_{1} + \ldots + r^{}_{n-1}$ as before. This simply gives
\begin{equation}\label{eq:gen-rec}
   \eta (r^{}_{1}, \ldots , r^{}_{n-1} ) \, = \, 
   \frac{(-1)^{r}}{2 + (-1)^{n+r-1}} \, \eta (0, \ldots, 0 ) 
\end{equation}
  for all $(r^{}_{1}, \ldots , r^{}_{n-1} ) \in \{ 0, 1 \}^{n-1}$,
  which establishes our claim.
\end{proof}

\begin{remark}\label{rem:only-two}
  It is clear from \eqref{eq:01} that $\eta (0, \ldots , 0 )$ is
  either $0$ (for $n$ odd) or $1$ (for $n$ even). The balanced
  odd-order correlations vanish accordingly, while the even-order
  correlations are thus fully determined by the value $\eta (0) = 1$
  from the $2$-point correlations.  \exend
\end{remark}

The result from Eq.~\eqref{eq:nmeanvalue} can be extended to get an
analogue of Theorems~\ref{thm:meanalpha} and \ref{thm:means} as
follows. First, observe that \eqref{eq:gen-rec} implies the relations
\[
\begin{split}
   \eta (4 m^{}_{1} , \ldots , 4 m^{}_{n-1}) \, & = \,
   \eta (m^{}_{1} , \ldots , m^{}_{n-1}) \\[1mm]
   \eta (4 m^{}_{1}  + 1, \ldots , 4 m^{}_{n-1} + 1) \, & = \,
   \myfrac{1}{4} \bigl( \eta (m^{}_{1}  +1 , \ldots , m^{}_{n-1} +1) 
   - (-1)^n\eta (m^{}_{1} , \ldots , m^{}_{n-1}) \bigr) \\[1mm]
   \eta (4 m^{}_{1} +2 , \ldots , 4 m^{}_{n-1} +2) \, & = \,
   - \myfrac{1}{2} \bigl( \eta (m^{}_{1} +1 , \ldots , m^{}_{n-1} +1)
   + (-1)^n \eta (m^{}_{1} , \ldots , m^{}_{n-1}) \bigr) \\[1mm]
   \eta (4 m^{}_{1} +3, \ldots , 4 m^{}_{n-1} +3) \, & = \,
   \myfrac{1}{4} \bigl( \eta (m^{}_{1} , \ldots , m^{}_{n-1})
   - (-1)^n \eta (m^{}_{1} +1 , \ldots , m^{}_{n-1} +1) \bigr).
\end{split} 
\]
Now, one arrives at the following result.

\begin{theorem}
   Let\/ $n\geqslant 2$ be fixed. Then, for any\/
   $\alpha > \frac{\log (3)}{\log (4)}$ and\/ $\beta > 0$, one has
\[ 
\begin{split}
    \lim_{N\to\infty} \myfrac{1}{N^{\alpha (n-1)}}
    \sum_{0 \leqslant m^{}_{1}, \ldots , m^{}_{n-1} \leqslant N\nts -1}
    \lvert \eta (m^{}_{1}, \ldots , m^{}_{n-1}) \rvert \, & = \, 0
    \quad\text{and}\quad \\[1mm]
    \lim_{N\to\infty} \myfrac{1}{N^{n-1}}
    \sum_{0 \leqslant m^{}_{1}, \ldots , m^{}_{n-1} \leqslant N\nts -1}
    \lvert \eta (m^{}_{1}, \ldots , m^{}_{n-1}) \rvert^{\beta}\,
    & = \, 0 \ts .
\end{split}    
\] 
\end{theorem}

\begin{proof}
  The above recursions, via the triangle inequality, give the analogue
  of \eqref{eq:esti} for general $n$. Now, setting
\[
    \Sigma (N) \, = \, \myfrac{1}{N^{\alpha (n-1)}}
      \sum_{0 \leqslant m^{}_{1}, \ldots , m^{}_{n-1} \leqslant N\nts -1}
    \lvert \eta (m^{}_{1}, \ldots , m^{}_{n-1}) \rvert^{\beta} ,
\]
we can repeat our previous estimates, with minor, but obvious
variations. Indeed, first setting $\beta = 1$, we can repeat the proof
of Theorem~\ref{thm:meanalpha}, which gives the first claim. Then,
setting $\alpha=1$, the second claim follows in complete analogy to
the proof of Theorem~\ref{thm:means}.
\end{proof}

\begin{example}
  As an example, beyond the standard $2$-point correlations, we
  consider the $4$-point correlations. In this case, let us define
\[
  \bs{\eta}(m^{}_1,m^{}_2,m^{}_3)\defeq \mbox{\tiny $\left(\begin{matrix}
        \eta(m^{}_1,m^{}_2,m^{}_3)\\
        \eta(m^{}_1,m^{}_2,m^{}_3+1)\\
        \eta(m^{}_1,m^{}_2+1,m^{}_3)\\
        \eta(m^{}_1,m^{}_2+1,m^{}_3+1)\\
        \eta(m^{}_1+1,m^{}_2,m^{}_3)\\
        \eta(m^{}_1+1,m^{}_2,m^{}_3+1)\\
        \eta(m^{}_1+1,m^{}_2+1,m^{}_3)\\
        \eta(m^{}_1+1,m^{}_2+1,m^{}_3+1)\\
\end{matrix}\right) $ }.
\]
Then, Proposition~\ref{prop:reduction} implies that 
\begin{equation}\label{eq:v4}
  \bs{\eta}(2m^{}_1+r^{}_1,2m^{}_2+r^{}_2,2m^{}_3+r^{}_3) 
  \, = \, \bs{B}_{(r^{}_1,r^{}_2,r^{}_3)} \ts
  \bs{\eta}(m^{}_1,m^{}_2,m^{}_3) \ts ,
\end{equation}
where 
\[
  \bs{B}_{(0,0,0)}\defeq\myfrac{1}{2}
  \mbox{\tiny
   $ \left(\begin{matrix}
        2&0&0&0&0&0&0&0\\
        -1&-1&0&0&0&0&0&0\\
        -1&0&-1&0&0&0&0&0\\
        1&0&0&1&0&0&0&0\\
        -1&0&0&0&-1&0&0&0\\
        1&0&0&0&0&1&0&0\\
        1&0&0&0&0&0&1&0\\
        -1&0&0&0&0&0&0&-1
\end{matrix}\right) $},\ \
\bs{B}_{(0,0,1)}\defeq\myfrac{1}{2}
\mbox{\tiny $\left(\begin{matrix}
-1&-1&0&0&0&0&0&0\\
0&2&0&0&0&0&0&0\\
1&0&0&1&0&0&0&0\\
0&-1&0&-1&0&0&0&0\\
1&0&0&0&0&1&0&0\\
0&-1&0&0&0&-1&0&0\\
-1&0&0&0&0&0&0&-1\\
0&1&0&0&0&0&0&1\\
\end{matrix}\right) $},
\]

\[
\bs{B}_{(0,1,0)}\defeq\myfrac{1}{2}
\mbox{\tiny  $ \left(\begin{matrix}
-1&0&-1&0&0&0&0&0\\
 1&0&0&1&0&0&0&0\\
 0&0&2&0&0&0&0&0\\
 0&0&-1&-1&0&0&0&0\\
 1&0&0&0&0&0&1&0\\
-1&0&0&0&0&0&0&-1\\
 0&0&-1&0&0&0&-1&0\\
 0&0&1&0&0&0&0&1\\
\end{matrix}\right) $},\ \ \bs{B}_{(0,1,1)}
\defeq\myfrac{1}{2}
\mbox{\tiny $ \left(\begin{matrix}
 1&0&0&1&0&0&0&0\\
0&-1&0&-1&0&0&0&0\\
 0&0&-1&-1&0&0&0&0\\
 0&0&0&2&0&0&0&0\\
-1&0&0&0&0&0&0&-1\\
0&1&0&0&0&0&0&1\\
 0&0&1&0&0&0&0&1\\
 0&0&0&-1&0&0&0&-1\\
\end{matrix}\right) $ }.
\]
Further, if $\bs{J}_8$ is the $8{\times} 8$ anti-diagonal matrix with
all ones on the anti-diagonal, one has
\begin{align*} 
\bs{B}_{(1,0,0)}\defeq\bs{J}_8\, \bs{B}_{(0,1,1)}\, \bs{J}_8,\quad
\bs{B}_{(1,0,1)}&\defeq\bs{J}_8\, \bs{B}_{(0,1,0)}\, \bs{J}_8\\
\bs{B}_{(1,1,0)}\defeq\bs{J}_8\, \bs{B}_{(0,0,1)}\, \bs{J}_8,\quad 
\bs{B}_{(1,1,1)}&\defeq\bs{J}_8\, \bs{B}_{(0,0,0)}\, \bs{J}_8.
\end{align*}
The implied relations from \eqref{eq:v4} are really a matrix version of
the general recursions from \eqref{eq:gen-rec}. Further, the sum matrix is
\[
  \bs{B}\,\defeq\sum_{r^{}_1,r^{}_2,r^{}_3\in\{0,1\}}
  \bs{B}_{(r^{}_1,r^{}_2,r^{}_3)}=\myfrac{1}{2}
  \mbox{\tiny  $
  \left(\begin{matrix}
      1&-1&-1&1&-1&1&1&-1\\
      0&0&0&0&0&0&0&0\\
      0&0&0&0&0&0&0&0\\
      0&0&0&0&0&0&0&0\\
      0&0&0&0&0&0&0&0\\
      0&0&0&0&0&0&0&0\\
      0&0&0&0&0&0&0&0\\
      -1&1&1&-1&1&-1&-1&1
\end{matrix}\right) $ },
\] 
which is an idempotent, that is, $\bs{B}^2=\bs{B}$.  There is a lot of
Thue--Morse structure in the explicit form of the matrices, as the
interested readers will have noticed. Here, we state a nice connection
with the matrices $\bs{E}^{}_0$, $\bs{E}^{}_1$ and $\bs{J}$ defined in
\eqref{eq:E-def}. It is quite clear that
$\bs{J}^{}_8=\bs{J}^{\otimes 3}$. Somewhat less clear is the
relationship between $\bs{B}_{(i,j,k)}$ and Kronecker products of the
matrices from \eqref{eq:E-def}. One can check that
\begin{align*}
  \bs{B}^{}_{(i,j,k)} \,
  & = \, \myfrac{1}{2}
    \bigl( \bs{E}^{}_i \otimes \bs{E}^{}_j \otimes \bs{E}^{}_k
    +\bs{J}^{\otimes 3}\ts (\bs{E}^{}_{1-i} \otimes \bs{E}^{}_{1-j}
    \otimes\bs{E}^{}_{1-k} ) \ts \bs{J}^{\otimes 3} \bigr) \\[1mm]
  & = \, \myfrac{1}{2} \bigl( \bs{E}^{}_i \otimes \bs{E}^{}_j
    \otimes  \bs{E}^{}_k + \bs{E}_{1-i}' \otimes \bs{E}_{1-j}'
    \otimes  \bs{E}_{1-k}' \bigr).
\end{align*}
We leave further details, including the validity of the generalisation
\[
  \bs{B}^{}_{(i_1,\ldots,i_{n-1})} \, = \, \myfrac{1}{2}
  \bigl( \bs{E}^{}_{i_1} \otimes \cdots \otimes \bs{E}^{}_{i_{n-1}}
    +(-1)^n\bs{E}_{1-i_1}'\otimes\cdots\otimes\bs{E}_{1-i_{n-1}}' \bigr),
\]
to the curious reader.  \exend
\end{example}

The relationship in Eq.~\eqref{eq:v4} suggests that one can study the
correlations $\eta( m^{}_1 , \ldots , m^{}_{n-1})$ via a related
regular sequence \cite{AS1992}. To do this, we start with the
generalisation of the relationship in \eqref{eq:v4},
\begin{equation}\label{eq:vBv}
  \bs{\eta}(2m^{}_1 + r^{}_1 , \ldots , 2m^{}_{n-1} + r^{}_{n-1}) \, = \,
  \bs{B}_{(r^{}_1 , \ldots , r^{}_{n-1})} \ts
  \bs{\eta}(m^{}_1 , \ldots , m^{}_{n-1}).
\end{equation}
For each $i\in\{0,1,\ldots,2^{n-1}-1\}$ with binary expansion
$(i)^{}_2 = r^{}_1 \ts r^{}_2 \cdots \ts r^{}_{n-1}$, we set
\[
    \bs{B}^{}_i \, = \, \bs{B}^{}_{(r^{}_1 , r^{}_2 , \ldots , r^{}_{n-1})}.
\]    
Now, we define the sequence $\eta_n:=(\eta_n(m))_{m\geqslant 0}$ by
\[
   \eta^{}_n (m) \, = \, \bs{e}_1^T
   \bs{B}^{}_{i^{}_0} \bs{B}^{}_{i^{}_1}\nts\cdots\ts
   \bs{B}^{}_{i^{}_{s}} \bs{e}^{}_1,
\]    
where $(m)^{}_n = i^{}_s \cdots i^{}_1 i^{}_0$ is the base-$n$
expansion of $m$, and $\bs{e}^{}_1$ is the standard column basis
vector of length $2^n$ with the $1$ in the first position.

Using the related sequence $\eta^{}_n$, we record the following result as
a refinement of \eqref{eq:nmeanvalue}. Here, instead of taking the
mean value over the $n$-dimensional integer cubes of the first orthant
as in \eqref{eq:nmeanvalue}, we traverse the integer points in that
orthant according to the order imposed by the relationship
\eqref{eq:vBv}.

\begin{theorem}
  Let\/ $n\geqslant 2$. Then, $\eta^{}_n$ has mean value zero, that
  is, 
\[
  \lim_{m\to\infty}\frac{1}{m}\sum_{j=0}^{m-1} 
  \eta^{}_n (j)  \, = \,  0 \ts .
\]  
\end{theorem}

\begin{proof}
  Note that, since all odd correlations vanish, without loss of
  generality, we assume that $n$ is even. Set
  $\Sigma(m)\defeq\frac{1}{m}\sum_{j=0}^{m-1} \eta^{}_n (j)$.
  Proposition~\ref{prop:reduction} gives that $\eta^{}_n (j)=O(1)$, so
  that
\[
   \Sigma(2^{n-1}m) \, \sim \, \Sigma(2^{n-1}m+a)
   \qquad \text{as $m\to \infty$}
\]
for any $a\in\{0,1,\ldots,2^{n-1}-1\}$. To prove the result, it thus
suffices to show that we have $\Sigma(2^{n-1}m)\xrightarrow{\quad} 0$
as $m\to\infty$. We use that, given by
Proposition~\ref{prop:reduction}, for even $n$ and any
$a\in\{0,1,\ldots,2^{n-1}-1\}$,
\[
  \eta^{}_{n}(2^{n-1}m+a) \, = \, \frac{t_a}{2}
  \big(\eta^{}_{n}(2^{n-1}m)+\eta^{}_{n}(2^{n-1}m+a)\big),
\]
where $t_a$ is the value of the Thue--Morse sequence at $a$ from
\eqref{eq:t-rec}. These recurrences are the direct generalisations of
those in \eqref{eq:v2recs}. With these in hand, we simply compute
\begin{align*}
  \Sigma(2^{n-1}m) \, & = \, \myfrac{1}{2^{n-1}}
                     \sum_{a=0}^{2^{n-1}-1}\myfrac{1}{m}
                     \sum_{j=0}^{m-1}\frac{t_a}{2}
                     \big(\eta^{}_{n}(j)+\eta^{}_{n}(j+a)\big)\\[1mm]
                   &=\,\myfrac{1}{2^{n-1}}\sum_{a=0}^{2^{n-1}-1}
                     t_a\biggl(\myfrac{1}{2m} \sum_{j=0}^{m-1}
                     \big(\eta^{}_{n}(j)+\eta^{}_{n}(j+a)\big)\biggr)\\[1mm]
                   &=\, O(m^{-1}) + \frac{\Sigma(m)}{2^{n-1}}\,
                     \sum_{a=0}^{2^{n-1}-1}t_a \ts ,
\end{align*}
where we have again used that $\eta^{}_{n}(j)$ is bounded to give the 
last equality. Since $\sum_{a=0}^{2^{n-1}-1}t_a=0$, we have that
$\Sigma(2^{n-1}m)=O(m^{-1})$, which proves the result.
\end{proof}

\section{Correlations for general weights}\label{sec:general}

We now consider the Thue--Morse {system} for general real weights,
$f(-1)$ and $f(1)$. Two quantities will be of paramount importance
here. First, we have
\[
   \EE(f) \, = \lim_{N\to\infty}\myfrac{1}{N}
   \sum_{m=0}^{N\nts -1} f(w^{}_m) \, =
   \int_{\XX}f(x_0)\dd\mu(x) \, = \,
   \frac{f(1)+f(-1)}{2} \ts ,
\]   
since the frequencies of $1$ and $-1$ in $w$ are both equal to
$1/2$. Second, set
\[
    h^{}_f \, \defeq \, \frac{f(1)-f(-1)}{2} \ts ,
\]
so that $f( \pm \ts 1)=\EE(f) \pm \ts h^{}_f$.

\begin{prop}\label{prop:f2and3}
  For any\/ $f \colon \{-1,1\}\xrightarrow{\quad} \RR$, we have\/
  $\eta^{}_{f} (m) = h_f^2\,\eta(m)+\EE(f)^2$ and
\[
  \eta^{}_{f} (m^{}_1 ,m^{}_2 ) \, = \, h_f^2\,\EE(f)
  \big(\eta(m^{}_1 ) + \eta(m^{}_2 ) + 
  \eta(| m^{}_1 - m^{}_2 | )\big) + \EE(f)^3.
\]
\end{prop}

\begin{proof}
  We use the identity
\[
  f(w^{}_n) \ts f(w^{}_{n+m}) \, = \, \big( f (w^{}_n)
   -\EE(f)\big) \big( f(w^{}_{n+m})-\EE(f)\big) + \EE(f)
   \big( f(w^{}_n) + f (w^{}_{n+m})\big)-\EE(f)^2 
\]
to give (with $S^m \! f \defeq f \circ S^m$)
\[
\begin{split}
  \eta_f(m) \, & = \int_{\XX}\big(f-\EE(f)\big)
  \big(S^m \! f-\EE(f)\big)\dd\mu+\EE(f)
  \int_{\XX}\big(f+S^m \! f \big)\dd\mu-\EE(f)^2 \\[2mm]
  & = \, h_f^2\,\eta(m)+\EE(f)\big(\EE(f)+
  \EE(f)\big)-\EE(f)^2=h_f^2\,\eta(m)+\EE(f)^2,
\end{split}
\]
which is the first desired result. The second result follows
similarly, recalling that $\eta(m^{}_1 ,m^{}_2)$ vanishes by
Corollary~\ref{coro:odd}.
\end{proof}

The point of Proposition~\ref{prop:f2and3} is to show that, in order
to calculate $\eta^{}_f$, one only needs to know the values of $\eta$
and so, by Corollary~\ref{coro:eta0} and Remark~\ref{rem:only-two},
one only needs to know the value $\eta(0)$. Indeed, in general, we
have
\begin{multline*}
  f(w^{}_k)\prod_{j=1}^{n-1} f(w^{}_{k+m_j}) \, = \,
  \big(f(w^{}_k) - \EE(f)\big)
  \prod_{j=1}^{n-1}\big( f(w^{}_{k+m_j}) - \EE(f)\big)\\
  -\sum_{j=0}^{n-1}\big(-\EE(f)\big)^{n-j}
  p^{}_j \bigl( f(w^{}_k) , f(w^{}_{k+m^{}_1}), \ldots ,
  f(w^{}_{k+m^{}_{n-1}}) \bigr),
\end{multline*}
where $p^{}_j (z^{}_1 , \ldots , z^{}_m)$ is the elementary symmetric
polynomial of degree $j$ in $m$ variables. Thus
\[
  \eta^{}_f (m^{}_1 , \ldots , m^{}_{n-1}) \, = \,
  h_f^n\,\eta (m^{}_1 , \ldots , m^{}_{n-1}) \,
  -\sum_{j=0}^{n-1} \big(-\EE(f)\big)^{n-j}
  \int_{\XX} p^{}_j ( f, S^{m^{}_1}\! f, \ldots , 
  S^{m^{}_{n-1}} \! f)\dd\mu.
\]
Noting that, for any $n \geqslant 2$, we have
\[
  \eta^{}_f (m^{}_1 , \ldots , m^{}_{n-1}) \, = \int_{\XX} f\cdot
  (S^{m^{}_1} \! f) \cdots (S^{m^{}_{n-1}} \! f) \dd\mu \ts
\]
gives the following result. 

\begin{theorem}
  For any\/ $n\geqslant 2$, the\/ $n$-point correlations of the\/
  $f$-weighted Thue--Morse system can be calculated from the
  balanced correlations. Consequently, they are ultimately derived
  from the single value\/ $\eta (0, \ldots, 0)$, which is\/ $0$ for\/
  $n$ odd and\/ $1$ for\/ $n$ even. \qed
\end{theorem}

This result can be made explicit as follows. {First,  observe}
that we have
\begin{equation}\label{eq:weg}
    \int_{\XX} x^{}_{\ell^{}_{0}} \ts x^{}_{\ell^{}_{1}} \! \cdots
    x^{}_{\ell^{}_{n-1}} \dd \mu (x) \, = \, \eta (
    \ell^{}_{1} - \ell^{}_{0}, \ell^{}_{2} - \ell^{}_{0}, \ldots , 
    \ell^{}_{n-1} - \ell^{}_{0} ) \, = \, 0
\end{equation}
for all \emph{even} $n$, as a result of Corollary~\ref{coro:odd},
because this refers to an odd-order correlation. Now, with
$f(x^{}_i) = \EE (f) + x^{}_i \ts h^{}_{f}$ and $m^{}_{0} \defeq 0$,
observe
\begin{align*}
   \eta^{}_{f} (m^{}_{1}, \ldots , m^{}_{n-1} ) \, & = 
   \int_{\XX} f(x^{}_{0}) \, f(x^{}_{m^{}_{1}}) \cdots
   f(x^{}_{m^{}_{n-1}}) \dd \mu (x)  
   \, = \int_{\XX} \, \prod_{i=0}^{n-1}  \bigl( \EE (f) +
   x^{}_{m_i} \ts h^{}_{f} \bigr) \dd \mu (x) \\[1mm]
   & = \, \EE (f)^{n} \, + \sum_{i=1}^{n}
   \EE (f)^{n-i} \, h^{i}_{f}  \int_{\XX}
   p^{}_{i} ( x^{}_{m^{}_{0}} , x^{}_{m^{}_{1}} , \ldots ,
   x^{}_{m^{}_{n-1}} ) \dd \mu (x) \ts ,
\end{align*}
with the elementary symmetric polynomials as above.  Here, the
integral over $p^{}_{i}$ vanishes whenever $i$ is odd, as a result of
\eqref{eq:weg}. A simple calculation now gives that the general
correlation coefficient $\eta^{}_{f} (m^{}_{1}, \ldots , m^{}_{n-1} )$
is equal to
\[
\EE (f)^{n} + \sum_{r=1}^{[\frac{n}{2}]}
    h^{2r}_{f}\, \EE (f)^{n-2r} \! \sum_{0 \leqslant 
   i^{}_{0} < i^{}_{1} < \cdots \ts < i^{}_{2r-1} \leqslant n-1}
   \! \eta( m^{}_{i^{}_{1}} \! - m^{}_{i^{}_{0}} , 
    m^{}_{i^{}_{2}} \! - m^{}_{i^{}_{0}} , \ldots ,
    m^{}_{i^{}_{2r-1}} \! - m^{}_{i^{}_{0}} ) \ts ,
\]
which expresses the general coefficients in terms of the balanced
ones.

In view of our above analysis, two comments are in order. On the one
hand, the various generalisations to Thue--Morse-like sequences
\cite{BGG-TM,BGbook} can and should be analysed, expecting analogous
results. On the other hand, it will be interesting to also look at
higher-order correlations in systems with absolutely continuous
spectrum, such as the Rudin--Shapiro (or Golay--Rudin--Shapiro)
sequence and its various generalisations \cite{CG2017,CGS2018,FMpre},
and to identify any crucial difference from the singular continuous
cases. See the work of Maz\'{a}\v{c}  \cite{Jan} for a first study in this direction.

\bigskip

\section*{Acknowledgements}

MB would like to thank Darren C.~Ong for inspiring discussions that
sparked our interest.  MC would like to express his gratitude to
Bielefeld University, where he visited for the first half of 2022 when
this research was done.  We thank Jan Maz\'{a}\v{c} and Aernout van
Enter for suggestions that helped to improve the manuscript.  This
work was supported by the German Research Foundation (DFG), within the
CRC 1283/2 \mbox{(2021-317210226)} at Bielefeld University.

\end{document}